\documentclass[12pt]{amsart}

\usepackage{amsmath,amsfonts,amsthm,amssymb}
\usepackage[margin=1in]{geometry}
\usepackage{graphicx}
\usepackage[bookmarks,colorlinks,citecolor=blue]{hyperref} 
\usepackage{tikz}
\usepackage{comment}
\usetikzlibrary{calc}
\theoremstyle{theorem}
\newtheorem{thm}{Theorem}[section]
\newtheorem{cor}[thm]{Corollary}
\newtheorem{lema}[thm]{Lemma}
\newtheorem{obs}[thm]{Proposition}

\newtheorem*{tthm}{Main Theorem}
\newtheorem{problem}[thm]{Problem}

\newtheorem{pr}{Algorithm}
\newtheorem{fact}[thm]{Fact}
\theoremstyle{definition} 
\theoremstyle{definition} \newtheorem{ex}{Example}[section] %

 \numberwithin{equation}{section}
 
 \theoremstyle{definition}
\newtheorem{con}[thm]{Conjecture}
\newtheorem{df}[thm]{Definition}
\newtheorem{exm}[thm]{Example}
\newtheorem{rem}[thm]{Remark}
\newtheorem{nots}[thm]{Notations}

\def\P{\mathbb{P}}

\def\Z{\mathbb{Z}}
\def\N{\mathbb{N}}

\def\C{\mathbb{C}}
\def\R{\mathbb{R}}
\def\Q{\mathbb{Q}}

\def\o{\mathcal{O}}

\DeclareMathOperator{\Spec}{Spec}

\DeclareMathOperator{\rank}{rank}

\def\CC{\mathbb{C}}
\def\ob{\begin{obs}}
\def\kob{\end{obs}}
\def\dow{\begin{proof}}
\def\kdow{\end{proof}}

\def\tw{\begin{thm}}
\def\ktw{\end{thm}}
\def\hip{\begin{con}}
\def\khip{\end{con}}
\def\lem{\begin{lema}}
\def\klem{\end{lema}}
\def\ex{\begin{exm}}
\def\prog{\begin{pr}}
\def\kprog{\end{pr}}
\def\wn{\begin{cor}}
\def\kwn{\end{cor}}
\def\uwa{\begin{rem}}
\def\kuwa{\end{rem}}
\def\kex{\end{exm}}
\def\dfi{\begin{df}}
\def\kdfi{\end{df}}
\def\fa{\begin{fact}}
\def\kfa{\end{fact}}

\title[tangential varieties
of the Segre--Veronese Varieties]{Cohen--Macaulay and Gorenstein tangential varieties\\
of the Segre--Veronese Varieties}
\author{M. Azeem Khadam}
\address{Institut f\"ur Mathematik und wissenschaftliches Rechnen, Karl-Franzens-Universit\"at Graz, NAWI Graz, Heinrichstraße 36, 8010 Graz, Austria}
\email{azeem.khadam@uni-graz.at}
\author{Martin Vodi\v{c}ka}
\address{Max Planck Institute for Mathematics in the Sciences, Inselstrasse 22, 041 03 Leipzig, Germany and University of Konstanz, Fachbereich Mathematik und Statistik, Konstanz, Germany}
\email{martin.vodicka@uni-konstanz.de}

\begin{document}
	

\begin{abstract}
	We classify the tangential varieties of the Segre--Veronese varieties which are Cohen--Macaulay or Gorenstein.
\end{abstract}

\maketitle

\section{Introduction}
Let $\P^N$ be the projective space over the complex ground field $\C$ and $X \subset \P^N$ a projective variety. We define the \emph{tangential variety} $\tau(X)$ of $X$ as the union of all the tangent lines to $X$. Likewise the \emph{(second) secant variety} $\sigma_2(X)$ is defined as the union of all the secant lines to $X$ together with all the points lying on $\tau(X)$. In this article our focus will be on tangential varieties of the Segre--Veronese varieties, {\bf so throughout this article we assume that $X$ is the Segre--Veronese variety}.

Investigating tangential and secant varieties is a part of classical algebraic geometry which was studied, among others, by Terracini and were brought into a modern light by F.L.~Zak~\cite{zak2005tangents}. One of the most basic questions is about the dimension, which can, in most of the cases, be calculated by Terracini Lemma (see \cite{terracini1911sulle} for the original statement and \cite{aadlandsvik1987joins, dale1984terracini} for modern versions). Another basic question is to provide a complete list of generators of the ideal of the tangential and secant varieties. In \cite{raicu2012secant} Raicu solved this problem for secant varieties and in \cite{oeding2014tangential} Oeding and Raicu obtained the analogous results for tangential varieties. In both cases their methods were based on representation theory.

Furthermore, $\sigma_2(X)$ is known to be normal \cite[Theorem 2.2]{vermeire2009singularities} and Cohen--Macaulay (which followed from \cite[Proposition 3.5]{khadam2020secant}). On the other hand, the tangential variety $\tau(X)$ is not always normal \cite[Example 2.22]{michalek2018flexible}. This is one of the crucial reason that investigating Cohen--Macaulay and Gorenstein properties of $\tau(X)$ was never easy and it remained an open problem to classify those tangential varieties of the Segre--Veronese varieties which are Cohen--Macaulay or Gorenstein. In \cite[Theorem 4.4]{khadam2020secant} the first author along with Mateusz Micha\l ek and Piotr Zwiernik classified those secant varieties of the Segre--Veronese varieties which are Gorenstein. In a special case the tangential variety $\tau(X)$ and the secant variety $\sigma_2(X)$ coincide with that of the locus of matrices of rank at most two. That is, its classification of Cohen--Macaulay or Gorenstein was classically known (see Remark~\ref{literature} for references and a discussion).

The purpose of this article is to present the complete classification of those tangential varieties of the Segre--Veronese varieties which are Cohen--Macaulay or Gorenstein. Here we state the main theorem of this article (see Theorem~\ref{cohen-macaulay-goren} for the proof). To this end, fix $k\in \N$ a positive integer and $\mathbf a,\mathbf b\in \N^k$ where $\mathbf a=(a_1,\ldots,a_k), \mathbf b=(b_1,\ldots,b_k)$ such that $a_i, b_i$ are positive integers. Let $X$ be the corresponding Segre--Veronese variety, i.e. the embedding of $\P^{b_1}\times\dots\times\P^{b_k}$ into $\P^N$ given by the very ample line bundle $\o(a_1,\ldots,a_k)$. 

\begin{tthm}
	The tangential variety of the Segre--Veronese variety is smooth if and only if
	\begin{itemize}
		\item [(S1)] $k=2$, $\mathbf a=(1,1)$, $b_1=1$, or
		\item [(S2)] $k=1$, $a=1$ or ($a=2$ and $b=1$).
	\end{itemize}
	
	If the tangential variety of the Segre--Veronese variety is not smooth, then it is Cohen--Macaulay if and only if one of the following holds
	\begin{itemize}
		\item [(CM1)] $k\geq3$, $\mathbf a=(1,\ldots,1)$,
	
		\item [(CM2)] $k=2$, $\mathbf a=(2,2)$, $\mathbf b=(1,1)$,
		
		\item [(CM3)] $k=2$, $\mathbf a=(1,2)$, $\mathbf b=(1,b_2)$ for all $b_2\geq 1$,
		
		\item [(CM4)] $k=2$, $\mathbf a=(1,1)$, $b_i>1$ for all $i=1,2$,
		
		\item [(CM5)] $k=1$, $a\geq3$, $b=1$,
		
		\item [(CM6)] $k=1$, $a=2$, $b>1$.
	\end{itemize}

	If the tangent variety of the Segre--Veronese variety is not smooth, then it is Gorenstein if and only if one of the following holds
	\begin{itemize}
		\item [(G1)] $k=3$, $\mathbf a=(1,1,1)$, $\mathbf b=(1,1,1)$,
		
		\item [(G2)] $k=2$, $\mathbf a=(1,2)$, $\mathbf b=(1,1)$,
		
		\item [(G3)] $k=2$, $\mathbf a=(1,1)$, $b_1=b_2$, $b_1>1$.
		
		\item [(G4)] $k=1$, $a\geq3$, $b=1$,
		
		\item [(G5)] $k=1$, $a=2$, $b$ is even.
	\end{itemize}
\end{tthm}

Note that the tangential variety is normal for the cases (S1 - S2). Other cases when the tangential variety is normal fall inside (CM1 - CM6). Precisely, the tangential variety of the Segre--Veronese variety is normal only in the following cases:
\begin{itemize}
	\item [(N1)] $k\geq1$, $\mathbf a=(1,\ldots,1)$ - the Segre case (cf. \cite[Proposition 8.5]{moz}),
	
	\item [(N2)] $k=1$, $a=2$, $b$ is arbitrary - the special Veronese case (see Remark~\ref{literature}).
\end{itemize}
This means we have \emph{non} normal tangential varieties of the Segre--Veronese varieties which are Cohen--Macaulay but not Gorenstein or which are Gorenstein (and hence Cohen--Macaulay as well) or which are not Cohen--Macaulay (hence not Gorenstein). Likewise, we have normal tangential varieties of the Segre--Veronese varieties which are not Gorenstein but this fact was already known (for Segre case (N1) by \cite[Theorem 8.9]{moz} and for special Veronese case (N2) see Remark~\ref{literature} for several references).

A part of the research on the geometry of tangential varieties has also been motivated by applications. In \cite{oeding2011set}, Oeding pointed out applications of the tangential variety of an $n$-factor Segre where the equations allow one to answer the question of membership for the following sets: the set of tensors with border rank 2 and rank $k \leq n$ (the secant variety is stratified by such tensors \cite{ballico2013tensor}), a special Context-Specific Independence model, and a certain type of inverse eigenvalue problem.

The investigation of properties like normal, Cohen--Macaulay or Gorenstein for the tangential (and secant) varieties remains a hard problem. The techniques from algebraic statistics together with toric geometry recently, somehow, made it possible to study them (see \cite{sturmfels2013binary} for a seminal paper). Our main result is based on methods from algebraic statistics, in particular cumulants, and toric geometry. The main idea behind cumulants is to treat points of the variety as (formal) probability distributions and apply methods from algebraic statistics \cite{zwiernik2012cumulants, pwz}. Cumulants have already been applied successfully on several occasions \cite{khadam2020secant, manivel2015secants, moz, michalek2018flexible, sturmfels2013binary}.

A change of coordinates, inspired by cumulants, leads to new structures on secant and tangential varieties. In our setting, cumulant methods turn the tangential variety of the Segre--Veronese variety locally into a toric variety, although, in general, the tangential variety is not a toric variety. This will be done in Proposition~\ref{toric covering}. Since Cohen--Macaulay and Gorenstein are local properties and our main object is locally a toric variety, we next apply methods from toric geometry. \emph{However}, it is not as easy as it looks to be, since in our case the tangential variety is \emph{not} normal and hence using toric geometry is highly nontrivial here.

Note that, according to our best knowledge, cumulant methods have so far applied only to normal varieties (or to decide when a variety is normal) and then to use the normal toric machinery to investigate above mentioned properties. So to use toric geometry we do not rely on the classical definition of a toric variety, where, in addition, the variety needs to be normal; see \cite[Chapter 13]{sturmfelsks} for discussion. We, in particular, use a criterion of Cohen--Macaulay and Gorenstein developed by Hoa and Trung \cite[Theorem 4.1]{trung1986affine}, which is equally applicable to non normal toric varieties.

This article is organized as follows. The first half of Section 2 is mainly based on \cite{khadam2020secant} where we recall the background results needed later. In particular, we define toric varieties and simplicial embeddings. For more general references for this and others (Cohen--Macaulay, Gorenstein rings etc), we recommend \cite{bruns1998cohen, cox2011toric, michalek2019invitation, cca, sturmfelsks}. In the second half, we present and elucidate \cite[Theorem 4.1]{trung1986affine} by using a few examples. In section 3, we study the toric geometry of our varieties and present the main theorem.\\

\noindent{\sc Acknowledgments.}
The authors are thankful to the reviewer for useful comments and suggestions. We would also like to thank Mateusz Micha\l ek for his guidance, as well as for important remarks and useful discussions about this article.

\section{Background results}

In this section, we recall a few definitions and results which we need to prove the main results of this article. Our approach is to use methods from \cite{khadam2020secant} where the authors mainly deal with the secant varieties of the Segre-Veronese varieties. More details and examples can be found in \cite{khadam2020secant}. Let $\N$ denote the set of nonnegative integers.
\begin{df}\label{toric-simplicial}
	
	(a) (Toric variety)
	Let $\mathbf x=(x_1,\ldots,x_N)$ and $\mathbf t=(t_1,\ldots,t_n)$, and $\mathcal C=\{\mathbf c_1,\ldots,\mathbf c_N\}$ be a fixed subset of $\mathbb N^n$. The set $\mathcal C$ defines a map $e_{\mathcal C}$ from $\C^n$ to $\C^N$ where $x_{i}=\mathbf t^{\mathbf c_i}:=t_1^{c_{i1}}\cdots t_n^{c_{in}}$ for $1\leq i\leq N$. The closure of the image of this map $V_{\mathcal C}:=\overline{e_{\mathcal C}(\C^n)}$ is called an \emph{affine toric variety}. We note that this differs from the classical definition of a toric variety, where, in addition, the variety needs to be normal; see \cite[Chapter 13]{sturmfelsks} for discussion.

	(b) (Simplicial complex) A \emph{simplicial complex} $\Delta$ on the vertex set $\{1,\ldots,n\}$ is a collection of subsets, called \emph{simplices}, closed under taking subsets, that is, if $\sigma\in \Delta$ and $\tau\subset \sigma$ then $\tau\in \Delta$. A simplex $\sigma\in \Delta$ of cardinality $|\sigma|=i+1$ has \emph{dimension} $\dim(\sigma)=i$. In this paper we allow vertices to have repeated labels. In this case $\{1,\ldots,n\}$ always refers to the labelling set of $\Delta$ rather than its vertex set; see \cite[Example 2.1]{khadam2020secant} for an example.

\end{df}

By the standard construction a simplicial complex defines an affine toric variety.
Let $\Delta$ be a simplicial complex with vertices labelled by variables $\mathbf t=(t_1,\dots,t_n)$ (with possible repetitions). Suppose $\Delta$ contains $N$ distinct simplices. Then $\Delta$ induces an embedding $e_\Delta:\CC^n\rightarrow\CC^N$, where the coordinates of the codomain are indexed by $\sigma\in \Delta$, by 
$$	\mathbf t\mapsto \mathbf x=(x_\sigma)_{\sigma\in \Delta},\qquad x_\sigma=\prod_{i\in \sigma} t_i.$$
By convention, the monomial corresponding to the empty set is $x_\emptyset=1$. If two simplices have exactly the same labels, as multisets, we may identify them.
We define the variety $V_\Delta\;:=\;\overline{e_\Delta(\C^n)}$.  Denote by $\Delta_{\geq 2}$ the set of simplices in $\Delta$ of dimension at least one. The toric variety $T_\Delta$ associated with the embedding $e_\Delta$ is the affine toric variety in $\C^{N-n-1}$ obtained as the closure of the projection of $V_\Delta$ to the coordinates $x_\sigma$ with $\sigma\in \Delta_{\geq 2}$.

We define the \emph{tangential variety} $\tau(Y)$ of a variety $Y$ as the union of all tangent lines to $Y$. Note that the tangential variety $\tau(V_\Delta)$ is parametrized by $2n$ parameters $t_i$ and $u_i$ for $1\leq i\leq n$:
\[
x_\sigma=1/n \left(\sum_{i \in \sigma} u_i \prod_{j \in \sigma\setminus\{i\}}t_j + (n-\dim\sigma - 1) \prod_{j \in \sigma} t_j \right) \qquad\mbox{for all }{\sigma\in \Delta}.
\]

 The following is one of the main results from \cite{khadam2020secant}.
\begin{thm}\label{main-tangent}
	The tangential variety of $V_\Delta$ is isomorphic to the product of $\CC^n$ and the variety $T_{\Delta}$. 
\end{thm}

Note that this Theorem~\ref{main-tangent} was mentioned in \cite[Theorem 2.9]{khadam2020secant}, but a proof was not included. So, we would like to inlcude the proof here, as it plays the central role in our investigations. To this end, we define an automorphism of $\C^N$ (originally was stated at the beginning of \cite[Section 2.2]{khadam2020secant}) as follows:
\[
y_\sigma\;=\;x_\sigma,\qquad \mbox{if }\dim (\sigma)\leq  0,\]
and otherwise, for all $\sigma\in \Delta_{\geq 2}$
\[
y_\sigma\;=\;\sum_{\sigma'\subseteq \sigma} (-1)^{\dim \sigma+\dim \sigma'}x_{\sigma'}\prod_{i\in \sigma\setminus \sigma'}x_{i},
\]
where the sum is taken over all subsimplices $\sigma'$ of $\sigma$ including the empty simplex, and the product is taken over vertices $i\in \sigma\setminus \sigma'$; see \cite[Example 2.4]{khadam2020secant} for an example.

\begin{proof}[Proof of Theorem~\ref{main-tangent}]
In the following we write $\sigma = v$ for $\sigma = \{v\}$, and regard $x_\sigma$ and $y_\sigma$ as polynomials in $t_i$ and $u_i$. Further, we claim that
$$
\begin{array}{ll}
	y_v\;=\; 1/n(u_v + (n-1)t_v) & \mbox{if }\dim(v)=0,\\
	y_\sigma\;=\;c\prod_{i\in \sigma}(u_i-t_i)& \mbox{if }\dim(\sigma)\geq 1
\end{array}
$$
for some constant $c$. Indeed, for $\dim(v) = 0$ it is clear since $y_v = x_v =  1/n(u_v + (n-1)t_v)$. For the case of $\dim(\sigma)\geq 1$, recall that $\Delta$ is labelled with repetitions, that is, we might have $t_i = t_j$ (likewise $u_i = u_j$) for two distinct vertices $i$ and $j$, but if we prove our claim of $y_\sigma$ for distinct variables $t_i, u_i$, then the claim in full generality will follow by specializing $t$'s and $u$'s. Hence, we may assume that for every $v\in \Delta$ we have associated two independent variables $t_v$ and $u_v$. In order to prove our claim, we need to prove that $u_v-t_v$ divides $y_\sigma$ for every $v\in \sigma$. Equivalently, we prove that if $u_v = t_v$, then $y_\sigma = 0$.

Let $u_v = t_v$ and define $\sigma'':=\sigma'\setminus\{v\}$ for some $v\in \sigma' \subseteq \sigma$. Then $x_v = u_v = t_v$ and we get, by using the parametrization of $\tau(V_\Delta)$, that
\[
x_{\sigma'}\prod_{j\in\sigma\setminus\sigma'}x_j = 1/n \left(\sum_{i \in \sigma'} u_i \prod_{j \in \sigma'\setminus\{i\}}t_j + (n-\dim\sigma' - 1) \prod_{j \in \sigma'} t_j \right)\prod_{j\in\sigma\setminus\sigma'}t_j=
\]
\[
= 1/n\left(t_v\prod_{j \in \sigma'\setminus\{v\}}t_j + \sum_{i \in \sigma'\setminus \{v\}} u_it_v \prod_{j \in (\sigma'\setminus \{v\})\setminus\{i\}}t_j+ (n-\dim\sigma' - 1) t_v\prod_{j \in \sigma'\setminus\{v\}}t_j\right)\prod_{j\in\sigma\setminus\sigma'}t_j
\]
\[
= 1/n\left( \sum_{i \in \sigma'\setminus \{v\}} u_i \prod_{j \in (\sigma'\setminus \{v\})\setminus\{i\}}t_j+ (n-\dim(\sigma'\setminus \{v\})-1)\prod_{j \in \sigma'\setminus\{v\}}t_j \right) \prod_{j\in(\sigma\setminus\sigma')\cup\{v\}}t_j = 
\]
\[
= x_{\sigma'\setminus\{v\}}\prod_{j\in(\sigma\setminus\sigma')\cup\{v\}}x_j = x_{\sigma''}\prod_{j\in\sigma\setminus\sigma''}x_j.
\]
Since the change of coordinates $y_\sigma$ are hierarchical in a sense that their formulas depend only on $x_{\sigma'}$ for $\sigma' \subseteq \sigma$, and the contributions of $x_{\sigma'}\prod_{j\in\sigma\setminus\sigma'}x_j$ and $x_{\sigma''}\prod_{j\in\sigma\setminus\sigma''}x_j$ cancel each other, therefore $y_\sigma = 0$. As $y_\sigma$ is divisible by $\prod_{i\in\sigma}(u_i-t_i)$ and $y_\sigma$ is a homogeneous polynomial of degree $\dim\sigma + 1$ in variables $u_i$ and $t_i$, hence our claim follows.

Finally, if we make a change of variables as $u_i' := u_i-t_i$, then for $\sigma \in \Delta_{\geq 2}$ $y_\sigma$ becomes a monomial which prametrizes $T_\Delta$. Also, it is easy to see that for $\dim(v) = 0$, $y_v$ parametrizes $\C^n$.

\end{proof}

\begin{rem}\label{cumulants}
	Note that the change of coordinates $y_\sigma$ from $\C^N$ to $\C^N$ was inspired by cumulants. $y_\sigma$ are in fact the cental change of coordinates which are required to get another coordinate changes from $\C^N$ to $\C^N$ called simplicial cumulants \cite[Definition 2.5]{khadam2020secant}, which were generalization of secant cumulants \cite{moz}. The secant variety of the Segre-Veronese variety locally turned out to be a toric variety in simplicial cumulants \cite[Remark 3.3]{khadam2020secant}. Interested reader can find more detail about cumulant coordinates in \cite{zwiernik2012cumulants, pwz}.
\end{rem}

Next we present an example of a simplicial complex whose associated toric variety is the (affine) Segre--Veronese variety --- our main object of investigations.

\begin{exm}\label{seg-ver}
	Fix $k\in \N$ a positive integer and $\mathbf a,\mathbf b\in \N^k$ where $\mathbf a=(a_1,\ldots,a_k), \mathbf b=(b_1,\ldots,b_k)$ such that $a_i, b_i$ are positive integers. Consider the vertex set $V=V_1\sqcup \cdots \sqcup V_k$, where each $V_i$ has $a_ib_i$ vertices such that for each $j=1,\ldots, b_i$, exactly $a_i$ vertices get labelled $t_{i,j}$. We denote $\Delta_{\rm SV}$ the simplicial complex with vertex set $V$. A subset $\sigma$ of $V$ forms a simplex of $\Delta_{\rm SV}$ if and only if $|\sigma\cap V_i|\leq a_i$ for all $i=1,\ldots,k$.
\end{exm}

If $\sigma\in \Delta_{\rm SV}$ then $\sigma=\sigma_1\sqcup\cdots\sqcup \sigma_k$, where each $\sigma_i$ is a multiset of labels $t_{i,j}$ with $|\sigma_i|\leq a_i$. Let $n=b_1+\cdots +b_k$ and $N$ be the number of simplices in $\Delta_{\rm SV}$. The toric embedding $e_{\Delta_{\rm SV}}:\C^n\to \C^N$ is given by 
$$
x_\sigma\;\;=\;\;\prod_{i=1}^k\prod_{j\in \sigma_i}t_{i,j}\qquad\mbox{for all}\quad\sigma\in \Delta_{\rm SV}.$$
The corresponding projective variety is obtained by introducing additional variables $t_{i,0}$ for $i=1,\ldots,k$ (the coordinates of each $\P^{b_i}$ are $(t_{i,0},\ldots,t_{i,b_i})$) and considering now a homogeneous parameterization $\P^{b_1}\times \cdots\times \P^{b_k}\to \P^{N}$
$$
x_\sigma\;\;=\;\;\prod_{i=1}^kt_{i,0}^{a_i-|\sigma_i|}\prod_{j\in \sigma_i}t_{i,j}\qquad\mbox{for all}\quad\sigma\in \Delta_{\rm SV}.$$
The image of this map is the \emph{Segre--Veronese variety} 
$$X\;:=\;v_{a_1}(\P^{b_1})\times\dots\times v_{a_k}(\P^{b_k}),$$ 
which is the embedding of the product $\P^{b_1}\times\dots\times\P^{b_k}$ given by the very ample line bundle $\o(a_1,\ldots,a_k)$. \textbf{From this point onward $\Delta_{\rm SV}$ will always be denoted by $\Delta$.}

\begin{rem}\label{local-coordinates}
	(a) The original affine variety $V_{\Delta}$ is isomorphic to the open subset of the Segre--Veronese variety obtained by setting $t_{i,0}\neq 0$ for all $i=1,\ldots,k$. This amounts to assuming $x_\emptyset\neq 0$. The Segre--Veronese variety can be covered by such varieties obtained by assuming that exactly one variable $t_{i,j}$ for each $i=1,\ldots,k$ is necessarily nonzero, or in other words, that a given coordinate $x_\sigma$ is nonzero. 
	
	(b) Consider the open subset of $\tau(X)$ given by $\tau(X)\cap \{x_\emptyset\neq 0\}$. On this subset $\tau(X)$ is isomorphic to the tangential variety of the affine variety $V_{\Delta}$. By Theorem~\ref{main-tangent} this (affine) tangential variety is isomorphic to the product of $\C^n$ and the variety $T_\Delta$ associated to the simplicial complex $\Delta$. This means that $\tau(X)$ can be covered by toric varieties, which is our main motivation to study the variety $T_{\Delta}$.
\end{rem}

Above Remark~\ref{local-coordinates}(b) gives us the following proposition:
\begin{obs}\label{toric covering}
	The tangential variety $\tau(X)$ is covered by toric varieties isomorphic to a product	of an affine space of dimension $n=\sum_{i=1}^{k}b_i$ and the toric variety $T_\Delta$.
\end{obs}

Note that the main difference between the secant (studied in \cite{khadam2020secant}) and tangential variety of the Segre-Veronese variety (which we study in this article) is that in case of secant variety, the associated toric varieties are always normal \cite[Proposition 3.5]{khadam2020secant}. Hence such toric varieties are Cohen-Macaulay and there is quite an easy criterion \cite[Proposition 8.2.12]{cox2011toric} to check whether they are Gorenstein. The authors of \cite{khadam2020secant} used this criterion to classify the Gorenstein property for secant varieties.

However, in our case, the associated toric variety $T_\Delta$ is not always normal, thus we can not apply this criterion. Instead, we employ \cite[Theorem 4.1]{trung1986affine}, which we recall below for readers' convenience, adapted to $T_\Delta=\Spec \C[S_\Delta]$. First we describe the semigroup $S_\Delta \subseteq \N^n$ associated to the toric variety $T_{\Delta}$, that is, when $T_{\Delta}=\Spec \C[S_\Delta]$. By following \cite{khadam2020secant} $S_\Delta$ is generated by those lattice points $(x_{i,j})\in \N^n$ which satisfy the inequalities
\begin{itemize}
	\item [(1)] $\sum_{j} x_{i,j}\leq a_i$ for all $i=1,\ldots,k$, and
	\item [(2)] $\sum_{i,j} x_{i,j}\geq 2$.
\end{itemize}

We now proceed with a few notations, which will be clarified below by using a list of examples, see Example~\ref{examples}.

\begin{nots}\label{not-goren}
	\begin{itemize}
		\item Let $G_\Delta$ denote the additive group in $\Z^n$ generated by $S_\Delta$ and put $r=\rank_\Z G_\Delta$.
		
		\item Let $C_\Delta$ denote the convex rational polyhedral cone spanned by $S_\Delta$ in $\Q^n_{\geq0}$. Hence $\dim_\Q C_\Delta=r$.
		
		\item Let $\mathcal{F}$ be the set of all facets of $C_\Delta$.
		
		\item For any facet $F\in\mathcal F$ let
		\[
		S_F=\{x=(x_{i,j})\in G_\Delta~|~ x+y\in S_\Delta \text{ for some } y=(y_{i,j})\in S_\Delta\cap F\}
		\]
		and $S'_\Delta=\bigcap_{F\in\mathcal F} S_F$. Note that $S_\Delta\subseteq S_F$ for every $F\in\mathcal F$.
		
		\item In our case, we will have only two kind of facets so we denote them by
		\begin{gather*}
			F_{i,j}:=\{x=(x_{i,j})~|~ x_{i,j}=0\}~~ \text{ and }~~ F_i:=\{x=(x_{i,j})~:~ \sum_j x_{i,j}= \sum_{l\neq i}\sum _j x_{l,j}\},
		\end{gather*}
		cf. Lemma~\ref{cone-facets}. We denote $S_{i,j}:=S_{F_{i,j}}$ and $S_i:=S_{F_i}$.
			
		\item For a subset $J$ of $\mathcal F$, we set $G_J=\bigcap_{F\not \in J}S_F \setminus \bigcup_{F' \in J}S_{F'}$. In particular, $$G_\mathcal{F}=G_\Delta \setminus \bigcup_{F' \in \mathcal{F}}S_{F'}.$$
		
		\item For a subset $J$ of $\mathcal F$, let $\pi_J$ be the simplicial complex of nonempty subsets $I$ of $J$ with the property $\bigcap_{F\in I}(S_\Delta\cap F)\not = \{0\}$.
	\end{itemize}
\end{nots}

We recall that a simplicial complex $\Delta$ is called \emph{acyclic} if the reduced homology group $\tilde{H}_q(\Delta; \C)$ vanishes for all $q\geq0$ (see \cite[Section 1.3]{cca} for basics on reduced homology groups). Moreover, let $n=b_1+\cdots +b_k$ and $\mathcal I=\{(i,j):1\leq i\leq k, 1\leq j\leq b_i\}$, and note that $|\mathcal I|=n$. The canonical unit vectors of $\mathbb R^n$ are denoted by $e_{i,j}$ where $(i,j)\in \mathcal I$ and $x=(x_{i,j})\in \R^n$. We follow the convention that the elements of $\mathcal I$ are ordered lexicographically. Furthermore, from this point onward, without loss of generality, we assume that $a_1\leq a_2\leq \ldots \leq a_k$.

\begin{thm}(\cite[Theorem 4.1]{trung1986affine})\label{hoa-trung}
	$\C[S_\Delta]$ is a Cohen--Macaulay (resp. Gorenstein) ring if and only if the following conditions are satisfied:
	\begin{itemize}
		\item[(i)] $S'_\Delta=S_\Delta$ (resp. $G_{\mathcal F} = x - S_\Delta$ for some $x\in G_\Delta$), and
		
		\item[(ii)] for every nonempty proper subset $J$ of $\mathcal F$, $G_J=\emptyset$ or $\pi_J$ is acyclic.
	\end{itemize}
\end{thm}

Let us look at a few examples.

\begin{exm}\label{examples}
	(1) Let $k=2$, $\mathbf a=(2,2)$, $\mathbf b=(1,1)$. Then
	\[
	S_\Delta = \N^2 \setminus \left(\{x\in 		\N^2 : 2\nmid x_{1,1}, x_{2,1}=0\}\cup \{x\in \N^2 : x_{1,1}		=0, 2\nmid x_{2,1}\}\right),
	\]
	$G_\Delta=\Z^2$, see Lemma~\ref{group} for general description of $G_\Delta$, and the cone $C_\Delta$ has two facets $F_{1,1}$ and $F_{2,1}$, see Lemma~\ref{cone-facets} for a more general facet description of our cone. Also,
	\[
	S_{1,1}=\{x\in \Z^2 : x_{1,1}>0\}\cup \{x\in \Z^2 : x_{1,1}		=0,2\mid x_{2,1}\}
	\] and
	\[
	S_{2,1}=\{x\in \Z^2 : x_{2,1}>0\}		
	\cup \{x\in \Z^2 : 2\mid x_{1,1},x_{2,1}=0\},
	\]
	and hence $S'_\Delta=S_{1,1}\cap S_{2,1}=S_\Delta$. Moreover,
	\[
	G_{\mathcal F}=\{x\in\Z^2 : x_{1,1}<0, x_{2,1}<0\}\cup \{(-1-2n,0) : n\in \N\}\cup \{(0,-1-2m) : m\in \N\}
	\]
	and there is no point $x\in \Z^2$ such that $G_{\mathcal F}=x-S_\Delta$, see Figure~\ref{ex(1)} below. Finally, for $I$ equal to $\{F_{1,1}\}$ or $\{F_{2,1}\}$, $\pi_I$ is a point and hence acyclic. Therefore $\C[S_\Delta]$ is Cohen--Macaulay but neither Gorenstein nor normal. This means that the tangential variety $\tau(v_2(\P^1)\times v_2(\P^1))$ is Cohen--Macaulay but neither Gorenstein nor normal.
	
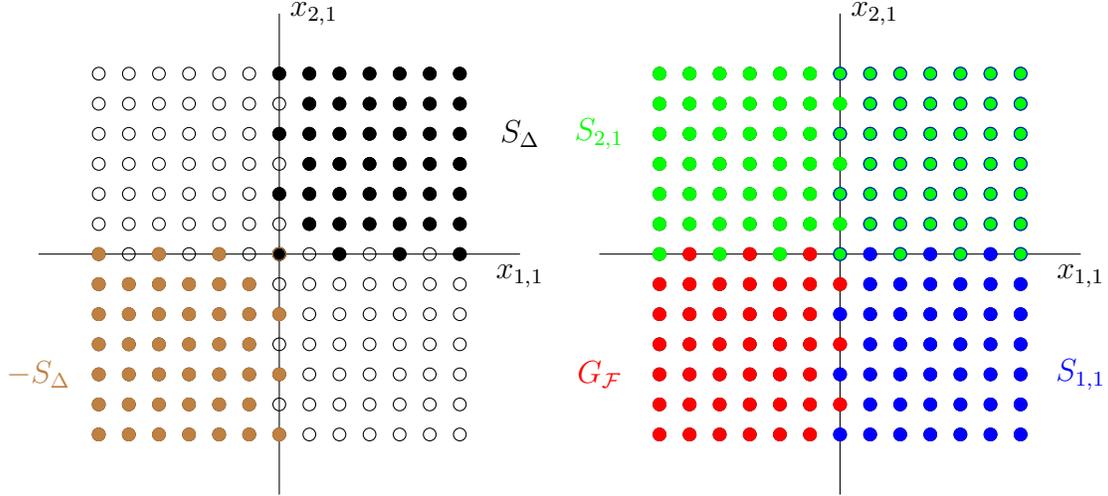
\begin{figure}[h]
\centering
\begin{tikzpicture} [scale=0.8]
\draw (-4,0) -- (4,0) node [anchor=north]{$x_{1,1}$};
\draw (0,-4) -- (0,4) node [anchor=west]{$x_{2,1}$} ;
\foreach \x in {-6,...,6}{
      \foreach \y in {-6,...,6}{
       \draw [black] (\x/2,\y/2) circle (3pt);;}}
\foreach \x in {1,2,...,6}{
      \foreach \y in {1,2,...,6}{
       \filldraw [black] (\x/2,\y/2) circle (3pt);;}}
\foreach \x in {0,1,2,3}{
\filldraw [black] (\x,0) circle (3pt);
\filldraw [black] (0,\x) circle (3pt);}
\filldraw [black] (4,2) circle (0pt) node {$S_\Delta$};
\foreach \x in {1,2,...,6}{
      \foreach \y in {1,2,...,6}{
       \filldraw [brown] (-\x/2,-\y/2) circle (3pt);;}}
\foreach \x in {1,2,3}{
\filldraw [brown] (-\x,0) circle (3pt);
\filldraw [brown] (0,-\x) circle (3pt);}
\draw [brown] (0,0) circle (3pt);
\filldraw [brown] (-4,-2) circle (0pt) node {$-S_\Delta$};

\end{tikzpicture}
\begin{tikzpicture} [scale=0.8]
\draw (-4,0) -- (4,0) node [anchor=north]{$x_{1,1}$};
\draw (0,-4) -- (0,4) node [anchor=west]{$x_{2,1}$} ;
\foreach \x in {-6,-5,...,6}{
      \foreach \y in {-6,-5,...,6}{
       \draw [black] (\x/2,\y/2) circle (3pt);;}}
\foreach \x in {-6,-5,...,6}{
      \foreach \y in {1,2,...,6}{
       \filldraw [green] (\x/2,\y/2) circle (3pt);;}}
\foreach \x in {-3,-2,...,3}{
\filldraw [green] (\x,0) circle (3pt);}
\filldraw [green] (-4,2) circle (0pt) node {$S_{2,1}$};  
\foreach \x in {-6,-5,...,-1}{
      \foreach \y in {1,2,...,6}{
       \filldraw [blue] (\y/2,\x/2) circle (3pt);;}}
       \foreach \x in {1,2,...,6}{
      \foreach \y in {1,2,...,6}{
       \draw [blue] (\y/2,\x/2) circle (3pt);;}}
\foreach \x in {-3,-2,-1}{
\filldraw [blue] (0,\x) circle (3pt);}
\foreach \x in {0,...,3}{
\draw [blue] (0,\x) circle (3pt);}
\foreach \x in {1,...,3}{
\draw [blue] (\x,0) circle (3pt);}
\foreach \x in {1,3,5}{
\filldraw [blue] (\x/2,0) circle (3pt);}

\filldraw [blue] (4,-2) circle (0pt) node {$S_{1,1}$}; 
\foreach \x in {-6,-5,...,-1}{
      \foreach \y in {-6,-5,...,-1}{
       \filldraw [red] (\x/2,\y/2) circle (3pt);;}}
\foreach \x in {-5,-3,-1}{
\filldraw [red] (\x/2,0) circle (3pt);
\filldraw [red] (0,\x/2) circle (3pt);}
\filldraw [red] (-4,-2) circle (0pt) node {$G_\mathcal F$};

\end{tikzpicture}
\caption{Example~\ref{examples}(1), $\mathbf a=(2,2)$, $\mathbf b=(1,1)$}
\label{ex(1)}
\end{figure}

	(2) Let $k=2$, $\mathbf a=(2,2)$, $\mathbf b=(1,2)$. Then $C_\Delta$ has three facets $F_{1,1}, F_{2,1}$ and $F_{2,2}$. Also, $G_\Delta=\Z^3$,
	$S_{1,1}=\{x\in \Z^3 : x_{1,1}>0\}\cup \{x\in \Z^3 : x_{1,1}=0, 	2\mid x_{2,1}+x_{2,2}\},$
	\begin{gather*}
		S_{2,1}=\{x\in \Z^3 : x_{2,1}\geq 0\}~~~~ \text{ and }~~~~ S_{2,2}=\{x\in\Z^3 : x_{2,2}\geq 0\}.
	\end{gather*}
	Therefore the point $e_{1,1}\in S'_\Delta\setminus S_\Delta$ and hence $\C[S_\Delta]$ is not Cohen--Macaulay. This means that the tangential variety $\tau(v_2(\P^1)\times v_2(\P^2))$ is not Cohen--Macaulay and hence neither Gorenstein nor normal. We in fact can generalize this example to any $b_2\geq2$, cf. Lemma~\ref{not-cm} (2). \\
	
	(3) Let $k=2$, $\mathbf a=(1,2)$, $\mathbf b=(1,1)$. Then 
	\[
	S_\Delta=\{x\in\N^2 : x_{1,1}\leq x_{2,1}\}\setminus \{x\in \N^2 : x_{1,1}=0, 2\nmid x_{2,1}\},
	\]
	$G_\Delta=\Z^2$, and $C_\Delta$ has two facets $F_{1,1}$ and $F_1$. Also,
	\begin{gather*}
		S_{1,1}=\{x\in \Z^2 : x_{1,1}>0\}\cup \{x\in \Z^2 : x_{1,1}=0, 2\mid x_{2,1}\}~ \text{ and }~
		S_1=\{x\in \Z^2 : x_{1,1}\leq x_{2,1}\},
	\end{gather*}
	and hence $S'_\Delta=S_{1,1}\cap S_1=S_\Delta$. Moreover,
	\[
	G_{\mathcal F}=\{x\in \Z^2 : x_{1,1}<0,x_{2,1}<0, x_{1,1}>x_{2,1}\}\cup \{(0,-1-2n) : n\in \N\}
	\]
	and hence $G_{\mathcal F}=(0,-1)-S_\Delta$, see Figure~\ref{ex(2)} below. Finally, for $I$ equal to $\{F_{1,1}\}$ or $\{F_1\}$, $\pi_I$ is a point and hence acyclic. Therefore $\C[S_\Delta]$ is Gorenstein (so is Cohen--Macaulay). Note that $\C[S_\Delta]$ is not normal. This means that the tangential variety $\tau(v_1(\P^1)\times v_2(\P^1))$ is Gorenstein but not normal.\\
	
	\begin{figure}[h]
\centering
\begin{tikzpicture} [scale=0.8]
\draw (-4,0) -- (4,0) node [anchor=north]{$x_{1,1}$};
\draw (0,-4) -- (0,4) node [anchor=west]{$x_{2,1}$} ;
\foreach \x in {-6,...,6}{
      \foreach \y in {-6,...,6}{
       \draw [black] (\x/2,\y/2) circle (3pt);;}}
\foreach \x in {1,2,...,6}{
      \foreach \y in {\x,...,6}{
       \filldraw [black] (\x/2,\y/2) circle (3pt);;}}
\foreach \x in {0,1,2,3}{
\filldraw [black] (0,\x) circle (3pt);}
\filldraw [black] (2,3.8) circle (0pt) node {$S_\Delta$}; 
\foreach \x in {1,2,...,6}{
      \foreach \y in {\x,...,6}{
       \filldraw [brown] (-\x/2,-\y/2) circle (3pt);;}}
\foreach \x in {1,2,3}{
\filldraw [brown] (0,-\x) circle (3pt);}
\draw [brown] (0,0) circle (3pt);
\filldraw [brown] (-2,-3.8) circle (0pt) node {$-S_\Delta$}; 
\filldraw [red] (-2,-4) circle (0pt) node { };

\end{tikzpicture}
\begin{tikzpicture} [scale=0.8]
\draw (-4,0) -- (4,0) node [anchor=north]{$x_{1,1}$};
\draw (0,-4) -- (0,4) node [anchor=west]{$x_{2,1}$} ;
\foreach \x in {-6,-5,...,6}{
      \foreach \y in {-6,-5,...,6}{
       \draw [black] (\x/2,\y/2) circle (3pt);;}}
\foreach \x in {-6,-5,...,6}{
      \foreach \y in {\x,...,6}{
       \filldraw [green] (\x/2,\y/2) circle (3pt);;}}
\filldraw [green] (-4,2) circle (0pt) node {$S_{1}$};  
\foreach \x in {0,...,5}{
      \foreach \y in {-6,...,\x}{
       \filldraw [blue] (\x/2+0.5,\y/2) circle (3pt);;}}
\foreach \x in {1,...,6}{
      \foreach \y in {\x,...,6}{
       \draw [blue] (\x/2,\y/2) circle (3pt);;}}
\foreach \x in {-3,-2,-1}{
\filldraw [blue] (0,\x) circle (3pt);}
\foreach \x in {0,...,3}{
\draw [blue] (0,\x) circle (3pt);}

\filldraw [blue] (4,-2) circle (0pt) node {$S_{1,1}$}; 
\foreach \x in {-5,...,-1}{
      \foreach \y in {-5,...,\x}{
       \filldraw [red] (\x/2,\y/2-0.5) circle (3pt);;}}
\foreach \x in {-5,-3,-1}{
\filldraw [red] (0,\x/2) circle (3pt);}
\filldraw [red] (-2,-3.8) circle (0pt) node {$G_\mathcal F$};

\end{tikzpicture}
\caption{Example~\ref{examples}(3), $\mathbf a=(1,2)$, $\mathbf b=(1,1)$}
\label{ex(2)}
\end{figure}
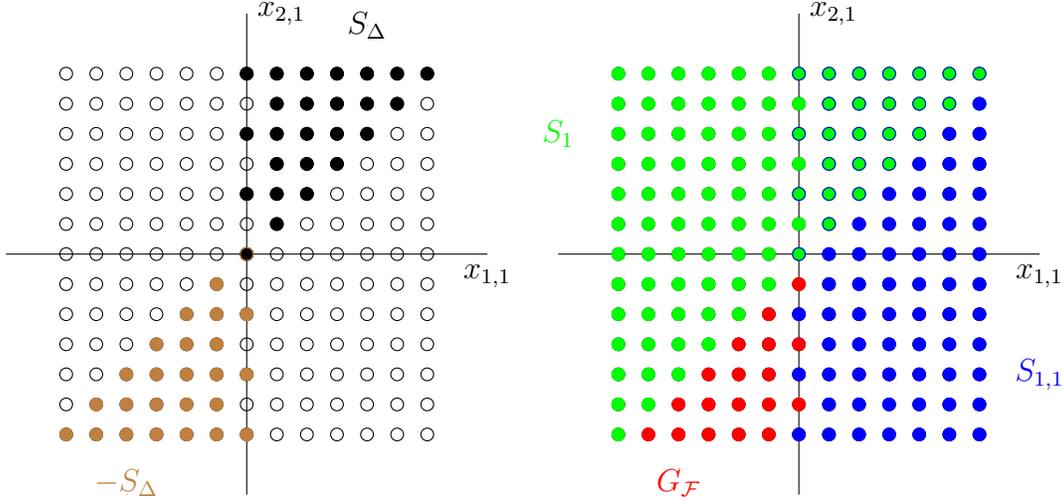
	
	(4) Let $k=2$, $\mathbf a=(1,2)$, $\mathbf b=(1,2)$. Then $C_\Delta$ has four facets $F_{1,1}, F_{2,1}, F_{2,2}$ and $F_1$. Also
	\begin{gather*}
		S_{1,1}=\{x\in \Z^3 : x_{1,1}>0\}\cup \{(0, x_{2,1}, x_{2,2}) : 2\mid x_{2,1}+x_{2,2}\},~~ S_{2,1}=\{x\in \Z^3 : x_{2,1}\geq 0\},
	\end{gather*}
	\begin{gather*}
		S_{2,2}=\{x\in \Z^3 : x_{2,2}\geq 0\}~ \text{ and }~ S_1=\{x\in \Z^3 : x_{1,1}\leq x_{2,1}+x_{2,2}\}.
	\end{gather*}
	We first claim that $S_\Delta$ satisfies condition (ii) of the above theorem. Indeed if $J$ is a singleton subset of $\mathcal F$, then clearly $\pi_J$ is acyclic. When $J$ has two elements, we need to consider all the cases separately:
	\begin{itemize}
		\item If $J=\{F_{1,1}, F_{2,1}\}$, then $\pi_J=\{\emptyset, \{F_{1,1}\},\{F_{2,1}\}, \{F_{1,1}, F_{2,1}\}\}$ 
		which is a simplex and hence acyclic. Note that $G_J=S_{2,2}\cap S_1 \setminus S_{1,1}\cup S_{2,1} \not= \emptyset$, since it for example contains $(-1,-1,5)$.
		
		\item If $J=\{F_{1,1}, F_{2,2}\}$, then $\pi_J=\{\emptyset, \{F_{1,1}\},\{F_{2,2}\}, \{F_{1,1}, F_{2,2}\}\}$ 
		which is acyclic. Note that $G_J=S_{2,1}\cap S_1 \setminus S_{1,1}\cup S_{2,2} \not= \emptyset$, since it for example contains $(-1,5,-1)$.
		
		\item If $J=\{F_{1,1}, F_1\}$, then $\pi_J=\{\emptyset, \{F_{1,1}\},\{F_1\}\}$ 
		which is \emph{not} acyclic, since\\
		$\tilde{H}_0(\pi_J; \C)\cong \C$. But in this case $G_J=S_{2,1}\cap S_{2,2} \setminus S_{1,1}\cup S_1 = \emptyset$. 
		
		\item If $J=\{F_{2,1}, F_{2,2}\}$, then $\pi_J=\{\emptyset, \{F_{2,1}\},\{F_{2,2}\}\}$ 
		which is \emph{not} acyclic, since\\
		$\tilde{H}_0(\pi_J; \C)\cong \C$. But in this case $G_J=S_{1,1}\cap S_1 \setminus S_{2,1}\cup S_{2,2} = \emptyset$. 
		
		\item If $J=\{F_{2,1}, F_1\}$, then $\pi_J=\{\emptyset, \{F_{2,1}\},\{F_1\}, \{F_{2,1}, F_1\}\}$ 
		which is acyclic. Note that $G_J=S_{1,1}\cap S_{2,2} \setminus S_{2,1}\cup S_1 \not= \emptyset$, since it for example contains $(1,-1,1)$.
		
		\item If $J=\{F_{2,2}, F_1\}$, then $\pi_J=\{\emptyset, \{F_{2,2}\},\{F_1\}, \{F_{2,2}, F_1\}\}$ which is acyclic. Note that $G_J=S_{1,1}\cap S_{2,1} \setminus S_{2,2}\cup S_1 \not= \emptyset$, since it for example contains $(1,1,-1)$.\\
		
		When $J$ has three elements, we need to consider all the cases separately:
		
		\item If $J=\{F_{1,1}, F_{2,1}, F_{2,2}\}$, then $\pi_J=\{\emptyset, \{F_{1,1}\},\{F_{2,1}\},\{F_{2,2}\}, \{F_{1,1}, F_{2,1}\}, \{F_{1,1}, F_{2,2}\}\}$ which is acyclic. Note that $G_J= S_1 \setminus S_{1,1}\cup S_{2,1}\cup S_{2,2} \not= \emptyset$, since it for example contains $(-2,-1,-1)$.
		
		\item If $J=\{F_{1,1}, F_{2,1}, F_1\}$, then $\pi_J=\{\emptyset, \{F_{1,1}\},\{F_{2,1}\},\{F_1\}, \{F_{1,1}, F_{2,1}\}, \{F_{2,1}, F_1\}\}$ which is acyclic. Note that $G_J= S_{2,2} \setminus S_{1,1}\cup S_{2,1}\cup S_1 \not= \emptyset$, since it for example contains $(-1,-4,1)$.
		
		\item If $J=\{F_{1,1}, F_{2,2}, F_1\}$, then $\pi_J=\{\emptyset, \{F_{1,1}\},\{F_{2,2}\},\{F_1\}, \{F_{1,1}, F_{2,2}\}, \{F_{2,2}, F_1\}\}$ which is acyclic. Note that $G_J= S_{2,1} \setminus S_{1,1}\cup S_{2,2}\cup S_1 \not= \emptyset$, since it for example contains $(-1,1,-4)$.
		
		\item If $J=\{F_{2,1}, F_{2,2}, F_1\}$, then $\pi_J=\{\emptyset, \{F_{2,1}\},\{F_{2,2}\},\{F_1\}, \{F_{2,1}, F_1\}, \{F_{2,2}, F_1\}\}$ which is acyclic. Note that $G_J= S_{1,1} \setminus S_{2,1}\cup S_{2,2} \cup S_1 \not= \emptyset$, since it for example contains $(1,-1,-1)$.
	\end{itemize}		
	Moreover, it is easy to check that $S'_\Delta=S_\Delta$, but there does not exist any $x\in \Z^3$ such that $G_{\mathcal F}=x-S_\Delta$. Therefore $\C[S_\Delta]$ is Cohen--Macaulay, but neither Gorenstein nor normal. This means that the tangential variety $\tau(v_1(\P^1)\times v_2(\P^2))$ is Cohen--Macaulay but neither Gorenstein nor normal (the same statement is true for any $b_2\geq2$, see Theorem~\ref{cohen-macaulay-goren}; see also Lemma~\ref{a=(1,2)}).\\
	
	(5) Let $k=1$, $a=3$, $b=1$ (same argument applies to any $a\geq3$). Then $S_\Delta=\N\setminus
	\{1\}$, $G_\Delta=\Z$ and $C_\Delta$ has only one facet $F_{1,1}$. Also, $S'_\Delta=S_{1,1}=S_\Delta$,
	\[
	G_{\{F_{1,1}\}}=\{-1-n : n\in \N\}\cup\{1\}=1-S_\Delta
	\]
	and condition (ii) of the previous Theorem~\ref{hoa-trung} trivially holds. Therefore $\C[S_\Delta]$ is Gorenstein (so is Cohen--Macaulay). Note that $\C[S_\Delta]$ is not normal. This means that the tangential variety $\tau(v_3(\P^1))$ is Gorenstein but not normal.\\
	
	(6) Let $k=1$, $a=2$, $b=2$. Then
	\begin{gather*}
		S_\Delta=\{x\in \N^2 : 2\mid x_{1,1}+x_{1,2}\},~~ G_\Delta=\{x\in \Z^2 : 2\mid x_{1,1}+x_{1,2}\}
	\end{gather*}
	and $C_\Delta$ has two facets $F_{1,1}, F_{1,2}$. Also,
	\begin{gather*}
		S_{1,1}=\{x\in G_\Delta : x_{1,1}\geq 0\}~~ \text{ and }~~ S_{1,2}=\{x\in G_\Delta : x_{1,2}\geq 0\},
	\end{gather*}
	and hence $S'_\Delta=S_{1,1}\cap S_{1,2}=S_\Delta$. Moreover, $G_{\mathcal F}=\{x\in G_\Delta : x_{1,1}<0, x_{1,2}<0\}$ and hence $G_{\mathcal F}=(-1,-1)-S_\Delta$. Finally, for $I$ equal to $\{F_{1,1}\}$ or $\{F_{1,2}\}$, $\pi_I$ is a point and hence acyclic. Therefore $\C[S_\Delta]$ is Gorenstein (see Theorem~\ref{cohen-macaulay-goren}(G5) for a generalization). Note that $\C[S_\Delta]$ is also normal. This means that the tangential variety $\tau(v_2(\P^2))$ is Gorenstein as well as normal.\\
	
	(7) Let $k=1$, $a=2$, $b=3$. Then
	\[
	S_\Delta=\{x\in \N^3 : 2\mid x_{1,1}+x_{1,2}+x_{1,3}\},~~ G_\Delta=\{x\in \Z^3 : 2\mid x_{1,1}+x_{1,2}+x_{1,3}\}
	\]
	and $C_\Delta$ has three facets $F_{1,1}, F_{1,2}$ and $F_{1,3}$. Also,
	\begin{gather*}
		S_{1,j}=\{x\in G_\Delta : x_{1,j}\geq 0\}~~ \text{ for }~~ j=1,2,3,
	\end{gather*}
	and hence $S'_\Delta=S_{1,1}\cap S_{1,2}\cap S_{1,3}=S_\Delta$. Moreover,
	\[
	G_{\mathcal F}=\{x\in G_\Delta : x_{1,1}<0, x_{1,2}<0, x_{1,3}<0\}
	\]
	which is not equal to $x-S_\Delta$, since the only possibility for $x$ is $(-1,-1,-1)$ which is not an element of $G_\Delta$. Finally, for any subset $I=\{F,F'\}$ of $\mathcal F$, $\pi_I=\{\emptyset, \{F\},\{F'\}, \{F,F'\}\}$ which is a simplex and hence acyclic. Therefore $\C[S_\Delta]$ is Cohen--Macaulay, but not Gorenstein (see Theorem~\ref{cohen-macaulay-goren}(CM6) and (G5) for a generalization). Note that $\C[S_\Delta]$ is also normal. This means that the tangential variety $\tau(v_2(\P^3))$ is Cohen--Macaulay and normal but not Gorenstein.
\end{exm}

\section{Cohen--Macaulay and Gorenstein tangential varieties}

In this section, we study the toric geometry of the toric variety $T_\Delta=\Spec \C[S_\Delta]$ and present the complete classification of those tangential varieties of the Segre--Veronese varieties which are Cohen--Macaulay or Gorenstein. We begin with the description of the group $G_\Delta$ when it is equal to the whole $\Z^n$.

\begin{lema}\label{group}
	We have $G_\Delta=\Z^n$, unless:
	\begin{itemize}
		\item [(i)] $k=2, \mathbf a=(1,1)$, when $G_\Delta=\{x\in\Z^n~:~\sum_j x_{1,j}=\sum_j x_{2,j}\}$,
		
		\item [(ii)] $k=1$, $a=2$, when $G_\Delta=\{x\in\Z^n~:~2\mid\sum_j x_{1,j}\}$, or
		
		\item [(iii)] $k=1$, $a=1$, when $G_\Delta=\{0\}$.
	\end{itemize}

\end{lema}
\begin{proof}
	We separately consider cases $k\geq 3$, $k=2$, and $k=1$.
	\medskip 
	
	{Case I. $k\geq 3$}
	
	Consider the set of vectors
	\begin{enumerate}
		\item $e_{1,1}+e_{i,j}$ for all $(i,j)\in \mathcal I$ with $i\neq 1$,
		
		\item $e_{1,j}+e_{2,1}$ for all $1<j\leq b_1$,
		
		\item $e_{2,1}+e_{3,1}$, and
		
		\item $e_{1,1}+e_{2,1}+e_{3,1}$
	\end{enumerate}
	that lie in $S_\Delta$. Now combining (3) and (4) we get $e_{1,1}\in G_\Delta$, and hence by using (1), we obtain $e_{i,j}\in G_\Delta$ for all $(i,j)\in \mathcal I$ with $i\neq 1$. Finally use (2) to get $e_{1,j}\in G_\Delta$ for all $1<j\leq b_1$, showing $G_\Delta=\Z^n$.
	\medskip
	
	{Case II. $k=2$}
	
	Consider the set of vectors
	\begin{enumerate}
		\item $e_{1,1}+e_{2,j}$ for all $1\leq j\leq b_2$,
		
		\item $e_{1,j}+2e_{2,1}$ for all $1\leq j\leq b_1$, and
		
		\item $2e_{2,1}$
	\end{enumerate}
	that lie in $S_\Delta$ if $a_2\geq 2$. Now combining (2) and (3) we get $e_{1,j}\in G_\Delta$ for all $1\leq j\leq b_1$, and hence by using (1), we obtain $e_{2,j}\in G_\Delta$ for all $1\leq j\leq b_2$, showing $G_\Delta=\Z^n$. If $a_2=1$ then the vectors of the form $e_{i_1,j_1}+e_{i_2,j_2}$ $(i_1\not= i_2)$ are the only generators of $S_\Delta$, and hence $G_\Delta=\{x\in\Z^n~:~\sum_j x_{1,j}=\sum_j x_{2,j}\}\neq \Z^n$.
	\medskip
	
	{Case III. $k=1$}
	
		Consider the set of vectors
	\begin{enumerate}
		\item $2e_{1,j}$ for all $1\leq j\leq b$, and
		
		\item $3e_{1,j}$ for all $1\leq j\leq b$
	\end{enumerate}
	that lie in $S_\Delta$ if $a\geq 3$. Combining (1) and (2), we get $e_{1,j}\in G_\Delta$ for all $1\leq j\leq b$, showing $G_\Delta=\Z^n$. If $a=2$ then all generators are of the form $e_{1,j_1}+e_{1,j_2}$, i.e. with the sum of the coordinates equal to two. It can be easily seen that they generate $$G_\Delta=\{x\in\Z^n~:~2\mid\sum_j x_{1,j}\}.$$ Finally, the case when $a=1$ is trivial.
\end{proof}

\begin{cor}\label{dim}
The tangential variety of the Segre-Veronese variety is of expected dimension $2n$, except of the case $k=2,\mathbf a=(1,1)$ when its dimension is $2n-1$ and the case $k=1,a=1$ when its dimension is $n$.
\end{cor}
\begin{proof}
The dimension of the toric variety $T_\Delta$ is the same as the dimension of the lattice $G_\Delta$, thus it follows directly from Lemma \ref{group}. Note that the dimension of the tangential variety is $2n$ also in the case (ii) from Lemma \ref{group}, although the group $G_\Delta$ is not equal to $\Z^n$ (instead it is isomorphic to $\Z^n$). 
\end{proof}

\begin{rem}
We suspect that the dimension of the tangential variety to the Segre-Veronese variety was already known since it can probably be derived by using Terracini Lemma~\cite{terracini1911sulle}. However, for the lack of reference, we also stated it in the form of Corollary~\ref{dim}. 
\end{rem}
 
We now have the description of the cone $C_\Delta$.
 
\begin{lema}\label{cone}
	The cone $C_\Delta$ is defined by the following set of inequalities:
	\begin{enumerate}
		\item $x_{i,j}\ge 0$ for all $(i,j)\in \mathcal I$, and
		
		\item $\sum_j x_{i,j}\le \sum_{l\neq i}\sum_j x_{l,j}$ for all $i$ such that $a_{i}=1$. 
		
	\end{enumerate}
\end{lema}

\begin{proof}
Let $C$ be the cone defined by the above inequalities $(1) - (2)$. It is easy to check that all generators of $S_\Delta$ lie in $C$ so $C_\Delta\subseteq C$. To prove the other inclusion, consider a point $x\in C\cap\Z^n$. It is sufficient to show that $2x\in C_\Delta$.
We, in fact, show a more general statement: any point $y=(y_{i,j})\in C\cap \mathbb Z^n$ with even sum of coordinates can be written as a sum of generators $(x_{i,j})$ of $S_\Delta$ with $\sum_{i,j} x_{i,j}=2$. We only need to prove this statement since it implies the lemma.

We denote $\sum_{i,j}y_{i,j}=2m$ and prove the statement by induction on $m$. For $m=0$ it is true. Consider the case $m>0$. If there exists an index $i_0$ such that all non-zero coordinates of $y$ are in the form $(i_0,j)$ then inequality $(2)$ for $i_0$ implies $a_{i_0}\neq 1$. Therefore $e_{i_0,j_1}+e_{i_0,j_2}$ are generators of $S_\Delta$ for all $1\le j_1,j_2\le b_{i_0}$. We can easily write $y$ as a sum of $m$ such generators.

Otherwise, we look at the inequalities $(2)$ for the point $y$ in which an equality holds, i.e. the point $y$ lies on the corresponding face of $C$. Since the inequalities $(2)$ is equivalent with
$$\sum_j y_{i,j}\le \frac 12 \sum_{i,j} y_{i,j}=m,$$
therefore we can have at most two indices $i$ for which the equality holds. Moreover, for every index $i$ for which the equality holds there exists a pair $(i,j)$ such that $y_{i,j}>0$ since $\sum_{j} y_{i,j}=m>0$.
So we pick two pairs $(i_1,j_1),(i_2,j_2)$ with $i_1\neq i_2$ such that $y_{i_1,j_1},y_{i_2,j_2}>0$, and for every index $i$ for which there is an equality in $(2)$ we have $i\in\{i_1,i_2\}$. This is clearly possible since there are at most two such indices $i$. The point $p=e_{i_1,j_1}+e_{i_2,j_2}$ is a generator of $S_\Delta$ and we claim that $z=y-p\in C$. We show that by checking all inequalities $(1)-(2)$. 

The inequalities $(1)$ obviously hold for the point $z$, hence we have to check the inequalities $(2)$ for every $1\le i\le k.$ Note that the inequalities $(2)$ for point $z$ are equivalent with
$$\sum_j z_{i,j}\le \frac 12 \sum_{i,j} z_{i,j}=m-1.$$
We distinguish two cases. 

\medskip
{Case I. $\sum_j y_{i,j} < m$.} 

In this case we have $$\sum_j z_{i,j}\le \sum_j y_{i,j}\le m-1,$$
and therefore the inequalities hold.

\medskip
{Case II. $\sum_j y_{i,j}=m$.}

By our definition of point $p$ we have $$\sum_j z_{i,j}= \sum_j y_{i,j}-\sum_j p_{i,j}= m-1.$$

We can conclude that $z\in C$ and $\sum_{i,j} z_{i,j}=2m-2$. By induction hypothesis we can write $z=y-p$ as a sum of generators which shows that $y$ can be written as sum of generators of $S_\Delta$ as well. This completes the proof.
\end{proof}

In order to employ Theorem~\ref{hoa-trung}, we require the description of the facets of the cone $C_\Delta$. To this end, let us first recall the following from Notations~\ref{not-goren}:
\begin{itemize}
	\item [(a)] $F_{i,j}=\{(x_{i,j})~:~ x_{i,j}=0\}$ for $(i,j)\in \mathcal I$, and
	
	\item [(b)] $F_i=\{(x_{i,j})~:~ \sum_j x_{i,j}= \sum_{l\neq i}\sum_j x_{l,j}\}$ for $1\leq i\leq k$.
\end{itemize}

We have the facet description of the cone $C_\Delta$:
\begin{lema}\label{cone-facets}

\begin{enumerate}
\item $F_{i,j}$ defines a facet of $C_\Delta$, unless:

\begin{itemize}
	\item [(i)] $k=3$, $a_1=a_2=1$, $a_3\geq2$ and $b_3=1$, when $F_{3,1}$ is not a facet,
	
	\item [(ii)] $k=3$, $a_1=a_2=a_3=1$ and $b_i=1$ for some $i$, when $F_{i,1}$ is not a facet,
	
	\item [(iii)] $k=2$, $a_1=1, a_2\geq2$, $b_2=1$, when $F_{2,1}$ is not a facet,
	
	\item [(iv)] $k=2$, $a_1=a_2=1$, when $F_{i,j}$ is not a facet for every $(i,j)\in \mathcal I$, or
	
	\item [(v)] $k=1$, $a=1$, when $C_\Delta=\{0\}$.
\end{itemize}

\item $F_i$ defines a facet of $C_\Delta$ for all $i$ such that $a_{i}=1$, unless:
\begin{itemize}
	\item [(i)] $k=2$, $a_1=a_2=1$, when $F_1=F_2=C_\Delta$, or
	
	\item [(ii)] $k=1$, $a=1$, when $C_\Delta=\{0\}$. 
\end{itemize} 
\end{enumerate}
\end{lema}

\begin{proof}
If $b_{i_0}\ge 2$, we show that $F_{i_0,1}$ forms a facet (the same proof applies to every $F_{i_0,j}$). We consider another semigroup $S_{\Delta'}$ which corresponds to $(\mathbf{a'},\mathbf{b'})$, where $\mathbf{a'}=\mathbf{a},$ $b'_{i_0}=b_{i_0}-1$ and $b'_i=b_i$ for all $i\not=i_0$. Moreover, we get coordinates of $S_{\Delta'}$ from those of $S_\Delta$ by skipping the coordinate $x_{i_0,1}$. Now there is a trivial bijection between the points of $S_\Delta$ which satisfy $x_{i_0,1}=0$ and the points of $S_{\Delta'}$. Therefore, $F_{i_0,1}$ defines a facet if and only if $S_{\Delta'}$ is full-dimensional, i.e. when $k=1, a\not=1$, or $k=2, \mathbf a'\not=(1,1)$, or $k\geq 3$, showing a part of $(iv)-(v)$ of the statement~(1).
If $b_{i_0}=1$ and $k\geq2$ (the case $k=1$ is obvious), we use the same argument as above for $k'=k-1$, $\mathbf{a'}=(a_1,\dots, \widehat{a_{i_0}}, \dots, a_k)$ and $\mathbf{b'}=(b_1,\dots, \widehat{b_{i_0}}, \dots, b_k)$, where $\widehat{\cdot}$ means we skip the corresponding coordinate. It gives us $(i)-(iii)$ and remaining part of $(iv)-(v)$ of the statement~$(1)$. This finishes the proof of (1).

To show that $F_{1}$ also forms a facet when $a_1=1$ (the same proof applies to every $F_i$), we need to find $n-1$ linearly independent points lying on it. We take the following points
\begin{enumerate}
	\item $e_{1,1}+e_{i,j}$ for all $(i,j)\in \mathcal I$ with $i>1$, and
	
	\item $e_{1,j}+e_{2,1}$ for all $1< j\leq b_1$.
\end{enumerate}
This also shows (i) of the statement~(2). Part (ii) is obvious and so this finishes the proof of (2).
\end{proof}

The following result tells us about holes inside $S_\Delta$.
\begin{lema}\label{holes}
\begin{itemize}
\item [(i)] If $a_i>2$, then the point $e_{i,j}$, for all $1\leq j\leq b_i$, belongs to $(C_\Delta\cap \Z^n)\setminus S_\Delta$, and

\item [(ii)] if $a_i=2$, then the points $e_{i,j}$, for all $1\leq j\leq b_i$, and $\sum_j c_je_{i,j}$ with $(c_j)\in \N^{b_i}$ such that $2\nmid\sum_j c_j$ belong to $(C_\Delta\cap \Z^n)\setminus S_\Delta$, unless $k=1$, $a=2$.
\end{itemize}
In other words, we have holes inside $S_\Delta$.

\end{lema}

\begin{proof}
All of the points listed clearly belong to $(C_\Delta\cap \Z^n)\setminus S_\Delta$. 

\end{proof}

In the following lemma, we study the condition (i) of Theorem~\ref{hoa-trung}.
\begin{lema}\label{not-cm}
\begin{enumerate}
	\item In the case $a_k\ge 3$ we have $S_\Delta\neq S'_\Delta$, unless $k=1$, $b=1$.

\item In the case $a_k= 2$ we have $S_\Delta\neq S'_\Delta$, unless:
\begin{itemize}
	\item [(i)] $k=2$, $\mathbf a=(2,2)$, $\mathbf b= (1,1)$,
	
	\item [(ii)] $k=2$, $\mathbf a=(1,2)$, $b_1=1$, or
	
	\item [(iii)] $k=1$.
\end{itemize} 
\end{enumerate}
\end{lema}
\begin{proof}
We show that the point $e_{k,1} \in S'_\Delta\setminus S_\Delta$. Clearly $e_{k,1}\not\in S_\Delta$, cf. Lemma~\ref{holes}. It remains to show that for every facet $F_{i}$ and $F_{i,j}$ of $C_\Delta$ we have $e_{k,1}\in S_{i}$ resp. $e_{k,1}\in S_{i,j}$, which implies that $e_{k,1}\in S'_\Delta$. 

For $a_{k}\geq2$ and for any $i$ such that $a_{i}=1$ we have $e_{k,1}+(e_{i,1}+e_{k,1})\in S_\Delta$ and $e_{i,1}+e_{k,1}\in S_\Delta\cap F_{i}$, therefore $e_{k,1}\in S_{i}$.
To show $e_{k,1}\in S_{i,j}$ we need to consider the cases $a_k>2$ and $a_k=2$ separately.
\bigskip

{Case I. $a_k>2$}

For any $(i,j)\neq (k,1)$ we have $e_{k,1}+2e_{k,1}\in S_\Delta$ and $2e_{k,1}\in S_\Delta\cap F_{i,j}$ which implies $e_{k,1}\in S_{i,j}$. For $S_{k,1}$, again we want to find a point $x\in S_\Delta\cap F_{k,1}$ such that $e_{k,1}+x\in S_\Delta$. We look at several cases:
\begin{itemize}
\item If $b_k\ge 2$ we can take $x=2e_{k,2}$.
\item If $a_{i}\ge 2$ for some $i\neq k$ we can take $x=2e_{i,1}$.
\item If $a_1=a_2=1$ we can take $x=e_{1,1}+e_{2,1}$.
\end{itemize}
We are left with the cases $k=1, b=1$ and $k=2, \mathbf{a}=(1,a_2), b_2=1$. However, in our statement we do not consider the first case and in the second case $F_{k,1}$ is not a facet by Lemma~\ref{cone-facets}, so we are done. This concludes Case I.
\bigskip

{Case II. $a_k=2$}

For the facet $F_{i_0,j_0}$ with $i_0\neq k$ we consider any pair $(i,j)\neq (i_0,j_0)$ with $i\neq k$. Note that such a pair does exist, unless $k=2$ and $b_1=1$. Then we have $e_{k,1}+(e_{k,1}+e_{i,j})\in S_\Delta$ with $e_{k,1}+e_{i,j}\in S_\Delta\cap F_{i_0,j_0}$, which implies that $e_{k,1}\in S_{i_0,j_0}$.
For the facet $F_{k,j}$ we again need to find a point $x\in S_\Delta\cap F_{k,j}$ such that $x+e_{k,1}\in S_\Delta$. We consider several cases:

\begin{itemize}
	\item If $a_{i}\ge 2$ for some $i\neq k$, then $x=2e_{i,1}$.
	\item If $a_1=a_2=1$, then $x=e_{1,1}+e_{2,1}$.
	\item If $\mathbf a=(1,2)$ and $b_2\ge 2$, then $x=e_{1,1}+e_{2,j_2}$ for some $j_2\neq j$.
\end{itemize}

To sum up we always can find such $x$, unless $k=1, a=2$ or $k=2, \mathbf a=(1,2), b_2=1$. The first case is excluded in the statement and in the second case $F_{2,1}$ is not a facet by Lemma~\ref{cone-facets} so we covered all cases. This concludes Case II, and hence completes the proof.
\end{proof}

In the following lemma, we study a special case of $k=2$.
\begin{lema}\label{a=(1,2)}
If $k=2$, $\mathbf{a}=(1,2)$, $b_1=1,b_2\ge 2$ then:
\begin{itemize}
\item[(i)] $S_\Delta=(C_\Delta\cap \Z^n)\setminus\{x\in C_\Delta\cap\Z^n~:~x_{1,1}=0, 2\nmid \sum x_{2,j}\}$,
\item[(ii)] $S_\Delta=S'_\Delta$,
\item[(iii)] for any proper subset $J$ of $\mathcal F$, $\pi_J$ is acyclic, unless $J=\{F_1,F_{1,1}\}$ or $J=\mathcal F\setminus \{F_1,F_{1,1}\}$, and
\item[(iv)] for $J=\{F_1,F_{1,1}\}$ or $J=\mathcal F\setminus \{F_1,F_{1,1}\}$, we have $G_J=\emptyset$. 
\end{itemize}
\end{lema}
\begin{proof}
The points $e_{2,j_1}+e_{2,j_2}$ are the only generators of $S_\Delta$ that lying on the facet $F_{1,1}$, therefore on this facet we have only the points with even sum of coordinates. Furthermore, from the proof of Lemma \ref{cone} we know that all points with even sum of coordinates in $C_\Delta\cap \Z^n$ also lie in $S_\Delta$. To prove (i) it remains to show that any point $x\in C_\Delta\cap \Z^n$ with odd sum of coordinates and $x_{1,1}>0$ is in $S_\Delta$.

It is easy to check that $x-e_{1,1}\in C_\Delta$ and therefore $x$ can be written as the sum of generators of $S_\Delta$ with the sum of coordinates equal to two. Since $x_{1,1}\le\sum_j x_{2,j}$, at least one of these generators must be in the form $e_{2,j_1}+e_{2,j_2}$. So we simply replace this generator by $e_{1,1}+e_{2,j_1}+e_{2,j_2}$ to write $x$ as the sum of generators of $S_\Delta$.

To prove (ii) we notice that we have $S_\Delta\subseteq S'_\Delta\subseteq C_\Delta\cap G_\Delta$. This, in fact, holds for any affine semigroup $S$. Thus, it is sufficient to show that for any point $x$ with $x_{1,1}=0$ and $2\nmid \sum x_{2,j}$ we have $x\notin S'_\Delta$. Indeed such a point $x\not\in S_\Delta'$ since $x\notin S_{1,1}$.

For part (iii) note that by Lemma \ref{cone-facets} we have $\mathcal F=\{F_{1,1},F_1\}\cup \{F_{2,j}: 1\le j\le b_2\}$. We claim that $\bigcap_{F\in I} (S_\Delta\cap F)=\{0\}$ if and only if $\{F_1,F_{1,1}\}\subseteq I$ or $\mathcal F\setminus \{F_1,F_{1,1}\}\subseteq I$. If $x\in S_\Delta\cap F_1\cap F_{1,1}$ then $x=0$, and hence $\bigcap_{F\in I} (S_\Delta\cap F)=\{0\}$. To prove the ``only if" part it is sufficient to consider maximal subset $I$ which does not contain two forbidden sets:

\begin{itemize}
\item For $I=\mathcal F\setminus\{F_{1,1},F_{2,j_0}\}$ we have $e_{1,1}+e_{2,j_0}\in \bigcap_{F\in I} (S_\Delta\cap F)$,
\item for $I=\mathcal F \setminus \{F_1,F_{2,j_0}\}$, we have  $2e_{2,j_0}\in \bigcap_{F\in I} (S_\Delta\cap F)$.
\end{itemize}
From this statement it follows that for $J=\{F_1,F_{1,1}\}$ or $J=\mathcal F\setminus \{F_1,F_{1,1}\}$ $\pi_J$ is not acyclic, because we have $\tilde{H_0}(\pi_J;\C)\cong\C$ or $\tilde H_{b_2-2}(\pi_J;\C)\cong\C$ respectively.
Moreover, it is straightforward to check that for any other set $J$ the complex $\pi_J$ is either a simplex, a union of two simplices with a common facet or a simplex without a facet, which are all acyclic.

For the part $(iv)$, we first consider the case $J=\{F_1,F_{1,1}\}$. Suppose on contrary that $x\in G_J$. Then the condition $x\in S_{2,j}$ for every $1\leq j\leq b_2$ implies that $x_{2,j}\ge 0$ for every $1\leq j\le b_2$. If now $x_{1,1}>0$ or $x_{1,1}=0$ and $2\mid \sum_jx_{2,j}$, then $x\in S_{1,1}$ which is not possible by the definition of $G_\Delta$. So $x_{1,1}< 0$ or $x_{1,1}=0$ and $2\nmid \sum_jx_{2,j}$. As $x\not\in S_1$, therefore $x_{1,1}> \sum_jx_{2,j}$ and hence $x_{2,j}<0$ for some $j$, which is a contradiction.
Now we consider the case $J=\mathcal F\setminus \{F_1,F_{1,1}\}$. Again we suppose on contrary that $x\in G_J$. Therefore $x_{2,j}<0$ for all $1\leq j\le b_2$. On the other hand, from $x\in S_{1,1}$ we get that $x_{1,1}\ge 0$, and hence $x_{1,1}-\sum_j x_{2,j}>0$. The last inequality implies that $x\notin S_1$ which is a contradiction. This completes the proof.
\end{proof}

The following lemma is about the non smoothness of $T_\Delta$.
\begin{lema}\label{nonsmooth}
\begin{itemize}
	\item [(i)] If $G_\Delta=\Z^n$, then $T_\Delta$ is not smooth.
	
	\item [(ii)] If $k=2$, $\mathbf a=(1,1)$, $b_i>1$ for all $i=1,2$, then $T_\Delta$ is not smooth.
	
	\item [(iii)] If $k=1$, $a=2$ and $b>1$, then $T_\Delta$ is not smooth.
\end{itemize}

\end{lema}

\begin{proof}
We know that an affine toric variety is smooth if and only if it is of the form $\Spec \C[C\cap G]$, where $C$ is the cone in the lattice $G$, and the rays of the cone $C$ form a basis of the lattice $G$ \cite[Theorem 1.3.12.]{cox2011toric}. 

For (i), in our setting, if $G_\Delta=\Z^n$ we have that $T_\Delta$ is smooth if the ray generators of $C_\Delta$ form a basis and $C_\Delta\cap \Z^n = S_\Delta$, which is not possible unless $a_i=1$ for all $i$, cf. Lemma~\ref{holes}. On the other hand, in this case all generators of $S_\Delta$ and therefore all ray generators of $C_\Delta$ are of the form $e_{i_1,j_1}+e_{i_2,j_2}$, where $i_1\not=i_2$. These generators lie in the sublattice of all points with even sum of coordinates and therefore can not form a basis of $\Z^n$. This finishes the proof of (i).

For (ii), note that all points of the form $e_{1,j_1}+e_{2,j_2}$ are generators of $S_\Delta$ and $G_\Delta$ as well as are ray generators of $C_\Delta$, so $C_\Delta \cap G_\Delta = S_\Delta$ and $\dim G_\Delta = b_1+b_2-1$. On the other hand, there are $b_1b_2$ points of form $e_{1,j_1}+e_{2,j_2}$, which satisfy $b_1b_2>b_1+b_2-1$, unless $b_1$ or $b_2$ equals 1. Therefore $T_\Delta$ can be smooth only if $b_1=1$ or $b_2=1$. This finishes the proof of (ii).

For (iii), note that $S_\Delta=\{x\in \N^b : 2\mid \sum_j{x_{1,j}}\}$, $G_\Delta=\{x\in \Z^b : 2\mid \sum_j{x_{1,j}}\}$ and hence $C_\Delta\cap G_\Delta = S_\Delta$. Further $2e_{1,j}$ are ray generators of $C_\Delta$ which are a basis of $G_\Delta$ only if $b=1$. This finishes the proof of (iii).
\end{proof}

Finally, we have the main result of this article.

\begin{thm}\label{cohen-macaulay-goren}
	The tangential variety of the Segre--Veronese variety is smooth if and only if
	\begin{itemize}
		\item [(S1)] $k=2$, $\mathbf a=(1,1)$, $b_1=1$, or
		\item [(S2)] $k=1$, $a=1$ or ($a=2$ and $b=1$).
	\end{itemize}
	
	If the tangential variety of the Segre--Veronese variety is not smooth, then it is Cohen--Macaulay if and only if one of the following holds
	\begin{itemize}
		\item [(CM1)] $k\geq3$, $\mathbf a=(1,\ldots,1)$,
	
		\item [(CM2)] $k=2$, $\mathbf a=(2,2)$, $\mathbf b=(1,1)$,
		
		\item [(CM3)] $k=2$, $\mathbf a=(1,2)$, $\mathbf b=(1,b_2)$ for all $b_2\geq 1$,
		
		\item [(CM4)] $k=2$, $\mathbf a=(1,1)$, $b_i>1$ for all $i=1,2$,
		
		\item [(CM5)] $k=1$, $a\geq3$, $b=1$,
		
		\item [(CM6)] $k=1$, $a=2$, $b>1$.
	\end{itemize}

	If the tangent variety of the Segre--Veronese variety is not smooth, then it is Gorenstein if and only if one of the following holds
	\begin{itemize}
		\item [(G1)] $k=3$, $\mathbf a=(1,1,1)$, $\mathbf b=(1,1,1)$,
		
		\item [(G2)] $k=2$, $\mathbf a=(1,2)$, $\mathbf b=(1,1)$,
		
		\item [(G3)] $k=2$, $\mathbf a=(1,1)$, $b_1=b_2$, $b_1>1$.
		
		\item [(G4)] $k=1$, $a\geq3$, $b=1$,
		
		\item [(G5)] $k=1$, $a=2$, $b$ is even.
	\end{itemize}
\end{thm}
\begin{proof}
It is sufficient to study the toric variety $T_\Delta=\Spec \C[S_\Delta]$, cf. Proposition~\ref{toric covering}. By applying Lemma~\ref{not-cm} the following are the only possible candidates for $\C[S_\Delta]$ to be Cohen--Macaulay and to be Gorenstein:
\begin{itemize}
\item [(i)] $a_i=1$ for all $1\leq i\leq k$,

\item [(ii)]$k=2$, $\mathbf a=(2,2)$, $\mathbf b= (1,1)$,

\item [(iii)] $k=2$, $\mathbf a=(1,2)$, $b_1=1$,

\item [(iv)] $k=1$, $a\geq 3$, $b=1$,

\item [(v)] $k=1$, $a=2$.

\end{itemize}
The candidates $(ii,iv)$ can be resolved by using Example~\ref{examples}(1) and Example~\ref{examples}(5) respectively, showing (CM2), (CM5), and (G4). For $k=1, a=1$, we have $S_\Delta=\{0\}$, which shows a part of (S2).

Now we consider candidate $(iii)$. The case $b_2=1$ is resolved by Example \ref{examples}(3), showing (G2) and a part of (CM3). Assume $b_2\ge 2$. By Lemma~\ref{a=(1,2)} both conditions from Theorem~\ref{hoa-trung} are satisfied, showing (CM3). Suppose that $T_\Delta$ is also Gorenstein. Then there exists $x_0\in\Z^n$ such that $G_\mathcal{F}=x_0-S_\Delta$. Since in $S_\Delta$ the point 0 has the minimal possible value in every coordinate of all points in $S_\Delta$, $x_0$ must be a point in $G_\mathcal{F}$ which has maximal value in every coordinate. Thus, $x_0=(0,-1,-1,\dots,-1)$ since clearly for all points $x\in\Z^n$, $x_{2,j}\ge 0$ implies $x\in S_{2,j}$ and $x_{1,1}>0$ implies $x\in S_{1,1}$. Note that the $(1,1)$-coordinate must be positive since there are holes in $S_\Delta$. However, we also have $y=(-1,\dots,-1)\in G_\mathcal{F}$ and $x_0-y=e_{1,1}\notin S_\Delta$ which is a contradiction.

We now consider candidate $(v)$. Note that in this case $S_\Delta=\{x\in \N^b : 2\mid \sum_j{x_{1,j}}\}$, $G_\Delta=\{x\in \Z^b : 2\mid \sum_j{x_{1,j}}\}$ and $F_{1,j}$ is a facet of $C_\Delta$ for all $1\leq j\leq b$. If $b=1$, then $S_\Delta\cong \N$ which gives part of (S2), hence completes (S2), cf. Lemma~\ref{nonsmooth}(iii). If $b>1$, then $S_{1,j}=\{x\in G_\Delta : x_{1,j}\geq 0\}$ for any $1\leq j\leq b$ and hence $S'_\Delta=\cap_jS_{1,j}=S_\Delta$, and $G_{\mathcal F}=\{x\in G_\Delta : x_{1,j}<0 \text{ for all } 1\leq j\leq b\}$. Now an element $x\in G_\Delta$ satisfying $G_{\mathcal F}=x-S_\Delta$ must be equal to $(-1,\dots,-1)$, which is only possible when $k$ is even. Moreover, for any proper subset $I$ of $\mathcal F$, $\pi_I$ is a simplex over $I$ which is well-known to be acyclic. This shows (CM6) and (G5).

Finally, we consider the candidate $(i)$. First assume that $k\ge 3$. In this case we have $S_\Delta=C_\Delta\cap G_\Delta$ so $T_\Delta$ is normal and therefore Cohen--Macaulay, showing (CM1).
Suppose it is also Gorenstein. That is, there exists some $x_0\in G_\Delta$ such that $G_{\mathcal F}=x_0-S_\Delta$. Since in $S_\Delta$ there is a unique point with the smallest sum of coordinates, therefore there is a unique point in $G_{\mathcal F}$ with the largest sum of coordinates. On the other hand, we have $$G_{\mathcal F}=\{x\in \Z^n|-x\in C_\Delta\setminus(\cup_i F_{i})\cup(\cup_{i,j} F_{i,j})\},$$ so this implies that there exists a unique lattice point in the interior of $C_\Delta$ with the smallest sum of coordinates. However, there is no point in $S_\Delta$ with the sum of coordinates one, and all point with the sum of coordinates two are of the form $e_{i_1,j_1}+e_{i_2,j_2}$ with $i_1\neq i_2$ which lie on $F_{i_1}$.

Further, by checking all inequalities we can see that the point $e_{1,1}+e_{2,1}+e_{3,1}$ is in the interior of $C_\Delta$ so it must be the unique lattice point with the sum of coordinates equal to three. However, if $k\ge 4$ then also the point $e_{1,1}+e_{2,1}+e_{4,1}$ is there and if $b_1\ge 2$ also the point $e_{1,2}+e_{2,1}+e_{3,1}$ is the interior lattice point in $C_\Delta$. This implies $k=3$ and $\mathbf b=(1,1,1)$. In this case, the point $(-1,-1,-1)$ has lattice distance one from all facets of $C_\Delta$ and therefore $G_{\mathcal F}=(-1,-1,-1)-S_\Delta$. Condition (ii) from Theorem \ref{hoa-trung} must be satisfied since we know that $T_\Delta$ is Cohen--Macaulay. This shows (G1).

Now we consider the case $\mathbf{a}=(1,1)$. We have $G_\Delta=\{x\in\Z^n;\sum x_{1,j}=\sum x_{2,j}\}\neq \Z^n$ and $S_\Delta=C_\Delta\cap G_\Delta$, so again $T_\Delta$ is normal, showing (CM4). If $b_1=1$, we have $S_\Delta=\{x\in \Z^n\vert x_{1,1}=\sum_j x_{2,j}\}$ and therefore $S_\Delta\cong \N^{b_2}$, which together with Lemma~\ref{nonsmooth}(ii) shows (S1).

Assume that $b_1>1$ and suppose that $T_\Delta$ is Gorenstein. As in the previous cases, this implies that there is a unique interior lattice point in $C_\Delta$ with the smallest sum of coordinates. The only candidate is the point $x=(1,1,\dots,1)$. From $x\in C_\Delta$ we have $b_1=\sum_j x_{1,j}=\sum x_{2,j}=b_2$. On the other hand, it is easy to check that in this case $G_{\mathcal F}=x-S_\Delta$. Moreover, condition (ii) from Theorem \ref{hoa-trung} holds since $T_\Delta$ is Cohen--Macaulay, showing (G3). This completes the proof. 

\end{proof}

We conclude the article with the following remark and a question thereafter.

\begin{rem}\label{literature}
		(1) The cases (S1) and (S2) are precisely those cases where the tangential variety $\tau(X)$ is the whole projective space $\P^N$.
		
		(2) For the case when $\Delta$ is a graph, that is, when the dimension of each of its simplices is at most one, then the tangential variety of $V_\Delta$ coincides with the secant variety of $V_\Delta$. In our setting, this occurs in the following two cases:
		\begin{itemize}
			\item [(i)] $k=1$, $a=2$ (special Veronese case), and
			
			\item [(ii)] $k=2$, $\mathbf{a} = (1,1)$ (Segre Case).
		\end{itemize}
		Therefore, in both of these cases, Cohen--Macaulay tangential varieties were classically known (as they are normal).
		On the other hand Gorenstein classification in these cases follows from \cite[Theorem 4.4]{khadam2020secant} (although it was known earlier as well, see \cite[Remark 4.5]{khadam2020secant} for references therein).
		
		(3) $(G1)$ was also proved in \cite[Theorem 8.9]{moz}.
\end{rem}

In this article we are able to completely classify those tangential varieties of the Segre-Veronese varieties which are Cohen-Macaulay and/or Gorenstein, that is, every local ring of every point in the projective variety is Cohen-Macaulay and/or Gorenstein respectively. However, there is also a different, but stronger property, named as arithmetically Cohen-Macaulay (respectively arithmetically Gorenstein), which asks when the ring is Cohen-Macaulay (respectively Gorenstein) after localizing at the zero point of the affine cone. The problem of classification of arithmetically Cohen-Macaulay and/or Gorenstein tangential varieties remains open.

\begin{problem}
	Classify tangential varieties of the Segre-Veronese varieties which are arithmetically Cohen-Macaulay and/or Gorenstein.
\end{problem}

\bibliographystyle{siam}
\bibliography{tan-var}

\end{document}